\documentclass[a4paper]{amsart}
\usepackage{amsmath,amsthm,amssymb,latexsym,epic,bbm,comment}
\usepackage{graphicx,enumerate,stmaryrd}
\usepackage[all,2cell]{xy}
\xyoption{2cell}

\usepackage[notref,notcite]{showkeys}
\usepackage[active]{srcltx}
\usepackage[parfill]{parskip}
\UseTwocells

\newtheorem{theorem}{Theorem}
\newtheorem{lemma}[theorem]{Lemma}
\newtheorem{remark}[theorem]{Remark}
\newtheorem{corollary}[theorem]{Corollary}
\newtheorem{proposition}[theorem]{Proposition}
\newtheorem{example}[theorem]{Example}

\newtheorem{definition}[theorem]{Definition}

\newcommand{\set}[1]{\{#1\}}
\newcommand{\leqJ}{\leq_{\mathcal{J}}}
\newcommand{\leqL}{\leq_{\mathcal{L}}}
\newcommand{\leqR}{\leq_{\mathcal{R}}}
\newcommand{\z}{\mathbf{z}}

\newcommand{\A}{\pP_{(1)}}
\newcommand{\B}{\pP_{(>1)}}

\newcommand{\pP}{\mathcal{P}}

\begin{document}

\title[Sub-bimodules of the identity for cyclic quivers]{Sub-bimodules of the identity bimodule\\ for cyclic quivers}

\author{Love Forsberg}


\begin{abstract}
We describe the combinatorics of the multisemigroup with multiplicities 
for the tensor category of subbimodules of the identity bimodule, 
for an arbitrary non-uniform orientation of a finite cyclic quiver.
\end{abstract}
\maketitle

\section{Introduction}\label{s1}

Modern 2-representation theory has its origins in \cite{BFK,Kh,CR,Ro}, see \cite{Ma} for an overview. 
A recent direction of $2$-representations theory, which started in the series 
\cite{MM1,MM2,MM3,MM4,MM5,MM6} and was further developed in \cite{GM1,GM2,Xa,Zh1,Zh2,CM,MZ,MMZ,MaMa,KMMZ,MT,MMMT,Zi},
studies $2$-analogues of finite dimensional algebras which are called \emph{finitary} 2-categories,
as was introduced in \cite{MM1}. Most of the natural examples of finitary 2-categories which appear in the
listed papers have an additional structure given by weak involution and adjunction morphisms which is
governed by the notion of \emph{fiat} 2-categories, which also appears in \cite{MM1}. Fiat 2-categories have particularly nice properties, as can be seen from the main results of the papers mentioned above.

Therefore, to understand properties of general finitary but not fiat 2-categories, it is important to
have more examples of such 2-categories. The papers \cite{GM1,GM2,Zh2} study interesting examples of 
finitary but not fiat 2-categories defined using sub-bimodules of the identity bimodule for path algebras
of tree quivers. Combinatorics of finitary 2-categories is described by the corresponding \emph{multisemigroup},
as defined in \cite{MM2}. We refer the reader to \cite{KM} for more details on multisemigroups.
As was explained in \cite{Fo}, based on rough ideas outlined in \cite{MM2},
the multisemigroup of a finitary 2-category is a shadow of a slightly more general structure 
called \emph{multisemigroup with multiplicities}. The latter contains enough information to fully
recover the algebra structure of the Grothendieck decategorification of a finitary 2-category. 
For the 2-categories studied in \cite{GM1,GM2,Zh2}, the multisemigroup structure was shown to 
be rather simple (degenerate) and given by a usual semigroup. This means that, in the examples 
studied in \cite{GM1,GM2,Zh2}, composition of two indecomposable 1-morphisms was always an 
indecomposable 1-morphism or zero.

The present paper studies the first family of examples which are similar in spirit to the ones studied
in \cite{GM1,GM2,Zh2}, but which lie outside the world of path algebras of tree quivers. More precisely, we 
study the finitary 2-category given by sub-bimodules of the identity bimodule for path algebras
of arbitrary orientations of the (finite) cyclic quiver. We show that, typically, the multisemigroup 
(with multiplicities) for this 2-category is a proper multisemigroup and not a semigroup. We also explicitly 
describe the combinatorics of the multiplication in this multisemigroup and all appearing multiplicities.
To ensure that we stay in the world of finite dimensional algebras, we assume that the orientation of 
our cyclic quiver is not uniform, that is, does not result in an oriented cycle. Equivalently, our
quiver is assumed to have at least one sink (and hence at least one source). The case with exactly
one sink turns out to be special and more difficult than the general case.

The paper is organized as follows: Section~\ref{s2} sets up our notation and contains all necessary 
preliminaries. In Section~\ref{s3} we analyze, in detail, indecomposable ideals and their product, for
any non-uniform orientation of an affine Dynkin quiver of type $\tilde{A}_n$. Here an important combinatorial
tool is the graph of an ideal which is defined in terms of both, maximal paths which the ideal contains
(these are vertices) and primitive idempotents that the ideal contains (these are edges). For example,
in Theorem~\ref{thm9} we show that connectedness of this graph controls indecomposability of the ideal.
We also describe product and tensor product of ideals in terms of intersection for the corresponding graphs.
Finally, in Section~\ref{s4} we present a number of combinatorial results, from various enumeration
formulae to a combinatorial criterion for the product of two indecomposable ideals to decompose.

\section{Notation and preliminaries}\label{s2}

Let $Q$ be a {\em quiver}. We denote by $Q_0$ the set of all {\em vertices} in $Q$ and by $Q_1$ the set of all 
{\em arrows} in $Q$. The function $\mathbf{h}:Q_1\to Q_0$ assigns, to each arrow $\alpha\in Q_1$,
its {\em head} $\mathbf{h}(\alpha)$. The function $\mathbf{t}:Q_1\to Q_0$ assigns, to each arrow $\alpha\in Q_1$,
its {\em tail} $\mathbf{t}(\alpha)$. 

Denote by $\pP^*=\pP_Q^*$ be the set of all (oriented) {\em paths} over $Q$. We also apply the head and tail 
terminology and notation to oriented paths, in the obvious way. Denote by $\pP=\pP_Q$ the set 
$\pP^*\cup \set{\z}$, where $\z$ is just a formal element. Then the set $\pP$ has the natural structure of a
semigroup with respect to {\em  concatenation} of paths. Here, two paths $\omega_1$ and $\omega_2$ can be concatenated 
if and only $\mathbf{h}(\omega_2)=\mathbf{t}(\omega_1)$. The corresponding concatenation
is denoted $\omega_1\omega_2$ and we have 
$\mathbf{h}(\omega_1\omega_2)=\mathbf{h}(\omega_1)$ and 
$\mathbf{t}(\omega_1\omega_2)=\mathbf{t}(\omega_2)$. If two paths cannot be 
concatenated in this way, their concatenation is defined to be $\z$. We also define the
concatenation of $\z$ with any other element to be $\z$. The element 
$\z$ becomes the zero element of the semigroup $\pP$ under this definition.
For convenience of the reader we will try to adopt the following convention:
\begin{itemize}
\item vertices of $Q$ will be denoted $x,y,\cdots$,
\item arrows of $Q$ will be denoted $\alpha,\beta,\cdots$,
\item paths over $Q$ will be denoted  $\omega,\upsilon$,
\end{itemize}
if necessary, with additional decoration such as indices. Unless $|Q_0|=1$, the semigroup $\pP$ does not have
any identity element. As usual, we denote by $\pP^1$ the semigroup $\pP$, in case $\pP$ has an identity element,
and, otherwise, the semigroup obtained from $\pP$ by attaching an external identity element.

Throughout the paper we fix an algebraically closed field $\Bbbk$. The main object
of study in the paper is the {\em path algebra} $\Bbbk Q$ of the quiver $Q$. This algebra can be defined
as the quotient of the semigroup algebra $\Bbbk[\pP]$ modulo the one-dimensional ideal $\Bbbk\z$.

We denote by $\leq_{\mathcal{J}}$ Green's {\em two-sided relation} on $\pP$, see \cite{Gr}, 
defined, for $\omega,\upsilon\in \pP$, via $\omega\leqJ\upsilon$ if and only if 
$\pP^1\omega\pP^1\supset \pP^1\upsilon\pP^1$. Similarly, for $\omega,\upsilon\in \pP$, 
we define Green's {\em left relation}  $\leq_{\mathcal{L}}$ on $\pP$  via $\omega\leq_{\mathcal{L}}\upsilon$ 
if and only if  $\pP^1\omega\supset \pP^1\upsilon$; furthermore, we define Green's 
{\em right relation}  $\leq_{\mathcal{R}}$ on $\pP$  via $\omega\leq_{\mathcal{R}}\upsilon$ 
if and only if  $\omega\pP^1\supset \upsilon\pP^1$. We denote by $\mathcal{J}$, $\mathcal{L}$
and $\mathcal{R}$ the equivalence relations describing the equivalence classes with respect to 
$\leq_{\mathcal{J}}$, $\leq_{\mathcal{L}}$ and $\leq_{\mathcal{R}}$. For $\omega \in \pP$,
we denote by $\mathcal{J}_{\omega}$ the $\mathcal{J}$-equivalence class containing $\omega$ and
use similar notation for $\mathcal{L}$ and $\mathcal{R}$.

We denote by $\mathfrak{l}:\pP^*\to \mathbb{N}$ the usual {\em length function} which assigns to 
a path the number of arrows which constitute this path (counted with all multiplicities). 
This function $\mathfrak{l}$ satisfies $\mathfrak{l}(\omega\upsilon)=\mathfrak{l}(\omega)+\mathfrak{l}(\upsilon)$, 
for all $\omega,\upsilon\in\pP^*$ such that $\omega\upsilon\neq\z$. Clearly, $\omega\leqJ\upsilon$ 
implies $\mathfrak{l}(\omega)\leq\mathfrak{l}(\upsilon)$. In particular, it follows that $\omega$ is 
the unique shortest path in $\pP^1\omega\pP^1$, Consequently, all $\mathcal{J}$-classes in
$\pP$ are singletons. This, of course, implies that all Green's equivalence relations are trivial, that is,
coincide with the equality relation,  and that $\leq_{\mathcal{J}}$, $\leq_{\mathcal{L}}$ and 
$\leq_{\mathcal{R}}$ are genuine partial orders (in general, they are only pre-orders). 

For each vertex $x\in\ Q_0$, there is the {\em trivial path} $\varepsilon_x$ at vertex $x$ of length $0$.
For this path, we have $x=\mathbf{h}(\varepsilon_x)=\mathbf{t}(\varepsilon_x)$ and the path
$\varepsilon_x$ serves as the local identity with respect to concatenation.  

A path $\omega$ of positive length is called a {\em cycle} provided that 
$\mathbf{h}(\omega)=\mathbf{t}(\omega)$. A cycle of length $1$ is called a {\em loop}.
The set $\pP$ is finite if, and only if, $Q$ is finite and contains no cycles. 

A path $\omega\in \pP^*$ will be called \emph{maximal} provided that it is a maximal element in
$\pP^*$ with respect to the restriction of the partial order $\leqJ$  from $\pP$ to $\pP^*$. 
We note that the restriction to $\pP^*$ is necessary as $\pP$ contains a maximum element
with respect to $\leqJ$, namely, the element $\z$. As both, $\leqL$ and $\leqR$, can be
extended, as partial orders, to $\leqJ$, it follows that any maximal path is a maximal 
element in $\pP^*$ both, with respect to $\leqL$ and $\leqR$, as well. However, the converse
is usually not true.

For a path $\omega$, denote by $Q_\omega$ the (not necessarily full) subquiver of $Q$ consisting of 
all vertices $x$ satisfying $\varepsilon_x\leqJ \omega$ and all arrows $\alpha:x\to y$ satisfying 
$\alpha\leqJ\omega$. The quiver $Q_\omega$ will be called the \emph{support} of $\omega$.

Denote by $\mathcal{T}$ the equivalence relations on $\pP^*$ defined by 
$\omega\mathcal{T}\upsilon$ if, and only if, $\mathbf{t}(\omega)=\mathbf{t}(\upsilon)$.
Denote by $\mathcal{H}$ the equivalence relations on $\pP^*$ defined by 
$\omega\mathcal{H}\upsilon$ if, and only if, $\mathbf{h}(\omega)=\mathbf{h}(\upsilon)$.
Extend both $\mathcal{T}$ and $\mathcal{H}$ to $\pP$ by letting $\z$ be its own equivalence class. 
Abusing notation, we will denote the resulting relations still by $\mathcal{T}$ and $\mathcal{H}$.
Let $\mathcal{Q}$ denote the intersection of $\mathcal{T}$ and $\mathcal{H}$, which is again
an equivalence relation. Similarly as for Green's relations, for $\omega\in\pP$, we denote
by $\mathcal{Q}_{\omega}$ the equivalence class of $\mathcal{Q}$ containing $\omega$, and we use
similar notation for $\mathcal{T}$ and $\mathcal{H}$. Define the sets:
\[\A=\set{\omega\in\pP^*\mid |\mathcal{Q}_{\omega}|=1}, \text{ and }\B=\pP\setminus \A.\]
As we do not have any relations in $\pP$ which would identify two different paths,
directly from the definitions it follows that the set $\B$ forms a two-sided ideal in 
the semigroup $\pP$.

Let $\omega\in\pP^*$. Then the equivalence class $\mathcal{H}_{\omega}$
contains a unique minimal element with respect to the partial order $\leqJ$,
namely $\varepsilon_{\mathbf{h}(\omega)}$. Moreover, the equivalence classes $\mathcal{T}_{\omega}$
contains a unique minimal element with respect to the partial order $\leqJ$,
namely $\varepsilon_{\mathbf{t}(\omega)}$. The set $\mathcal{H}_{\omega}\cup\set{\z}$
is a left ideal of the semigroup $\pP$ and the set $\mathcal{T}_{\omega}\cup\set{\z}$
is a right ideal of $\pP$.

From the fact that $\B$ is an ideal in $\pP$, we have the following:
for $\omega,\upsilon\in\pP^*$, the facts that $\omega\leqJ\upsilon$ and $\upsilon\in\A$ imply that 
$\omega\in\A$ as well. Consequently, the fact that all \emph{maximal} paths are in $\A$ implies
that all paths are in $\A$. Note that each path $\varepsilon_x$, where $x\in Q_0$, is always in $\A$.

Being the zero element of $\pP$, the element $\z$ belong to any left, any right and any 
two-sided ideal  of $\pP$. We will work both with semigroup ideals in $\pP$ and algebra ideals
in $\Bbbk Q$. To avoid any confusions, we will always specify what kind of ideals we consider at each particular
moment.  Note that, any semigroup ideal ${I}\subset\pP$ cab be linearized to give 
an algebra ideal $I_{\Bbbk}:=\Bbbk {I}/\Bbbk\z$. Moreover, the map $I\mapsto I_{\Bbbk}$ is, clearly, 
injective. It is natural to ask under which conditions this map is bijective. We answer this
in Proposition~\ref{prop2} below. As a first step, we will need the following lemma.

\begin{lemma}\label{lem1}
Let $I\subset \Bbbk Q$ be a non-zero algebra ideal and $a\in I$ be of the form 
\[a=\sum_{\omega\in\pP^*}c_\omega \omega,\quad\text{ where }c_{\omega}\in\Bbbk.\]
Then, for every $\omega\in\pP^*$ satisfying $\omega\in\A$ and $c_\omega\neq 0$, we have $\omega\in I$. 
\end{lemma}

\begin{proof}
Fix some $\omega\in \A$ such that $c_\omega\neq 0$. Let $h=\mathbf{h}(\omega)$ and $t=\mathbf{t}(\omega)$. 
Then $\varepsilon_h a\varepsilon_t$ is a sum of elements of the form $c_{\upsilon}\upsilon$, where 
$\upsilon$ is a path such that $\mathbf{h}(\upsilon)=h$ and $\mathbf{t}(\upsilon)=t$. 
Because of our assumption that $\omega\in \A$, we thus have $\varepsilon_h a \varepsilon_t=c_\omega\omega$. 
As $a\in I$, it follows that $c_\omega^{-1}\varepsilon_h a \varepsilon_t=\omega\in I$ as well. The claim
follows.
\end{proof}

\begin{proposition}\label{prop2}
Let $Q$ be a quiver such that $\pP^*=\A$. Then every algebra ideal $I$ in $\Bbbk Q$ is  of the form
$\hat{I}_{\Bbbk}$, for some semigroup ideal $\hat{I}$ in $\pP$.
\end{proposition}

\begin{proof}
For an algebra ideal $I$ in $\Bbbk Q$, define 
\[X:=\set{\omega\in\pP^*\mid \exists a\in I: a=\sum_{\upsilon\in\pP^*}c_{\upsilon}\upsilon\text{ and } c_\omega\neq 0}.\]
Directly from the definition we have $I\subset \Bbbk X$. At the same time, from Lemma~\ref{lem1} we have 
$X\subset I$ and hence $I=\Bbbk X$. Now let $\hat{I}=X\cup\set{\z}$. We claim that $\hat{I}$ is a semigroup 
ideal of $\pP$ and that $I=\hat{I}_{\Bbbk}$. The second part of the claim follows from the first part of the 
claim and all definitions. To prove the first part of the claim, we note that the standard basis of  $\Bbbk Q$ 
consists of elements in $\pP^*$. These elements multiply just as in $\pP$ with the only difference
that the semigroup element $\z$ is changed to the algebra element $0$. Therefore the fact that $I$ is an 
ideal implies that $\hat{I}$ is an ideal.
\end{proof}

Proposition~\ref{prop2} applies, in particular,  to all trees, regardless of orientation. If $Q$ is an orientation
of a cycle, then Proposition~\ref{prop2} applies if and only if $Q$ has at least two sinks (or, equivalently,
at least two sources).

Let $Q$ be a finite quiver without oriented cycles. Denote by $\mathrm{max}(Q)$ the quiver defined as follows:
\begin{itemize}
\item the set $\mathrm{max}(Q)_0$ of vertices in $\mathrm{max}(Q)$ is defined to be the set of all 
sources and all sinks in $Q$;
\item for $x,y\in \mathrm{max}(Q)_0$, the set of arrows in $\mathrm{max}(Q)$ from $x$ to $y$ coincides
with the set of all maximal paths in $Q$ from $x$ to $y$.
\end{itemize}
The underlying undirected graph of $\mathrm{max}(Q)$ has the natural structure of a bipartite graph in which 
we collect all sources in one part of the graph and  all sinks in the other part of the graph. 

\section{Oriented $\tilde{A}_n$ quivers}\label{s3}

\subsection{Sinks and sources}\label{s3.1}

From now on, we specialize to quivers whose underlying undirected graph is the affine Dynkin diagram 
$\tilde{A}_{n}$
\begin{displaymath}
\xymatrix{\bullet\ar@{-}[r]\ar@{-}@/^0.7pc/[rrrrr]
&\bullet\ar@{-}[r]&\bullet\ar@{-}[r]&\dots\ar@{-}[r]&\bullet\ar@{-}[r]&\bullet} 
\end{displaymath}
with $n$ vertices, for some $n\geq 1$. Moreover we require that the orientation of  $Q$ is chosen such
that $Q$ is acyclic, in particular, $n\neq 1$, and, furthermore,  
$Q$ has at least one source and at least one sink. We will call such orientation {\em admissible}.
Let $k$ be the 
number of sources in $Q$. Then the number of sinks in $Q$ is also $k$ because sources and sinks alternate
when one goes around our underlying unoriented cycle of the Dynkin diagram. By assumption, we have 
$k\geq 1$ and $|Q_0|=n$.

\begin{example}\label{exm3}
{\rm
Here is an example with $n=8$ vertices and $k=3$ sinks. Sources and sinks are highlighted. 
In this particular example we have $\pP^*=\A$.
\[\xymatrix{&\mathbf{1}&\mathbf{2}\ar[l]\ar[dr]\\
\mathbf{8}\ar[ur]\ar[d]&&&\mathbf{3}\\
\mathbf{7}&&&\mathbf{4}\ar[u]\ar[dl]\\
&6\ar[ul]&5\ar[l]}\]
}
\end{example}

\begin{example}\label{exm4}
{\rm
Here is another example with five vertices and one sink (again, sources and sinks are highlighted). 
In this examples there are two different paths from $\mathbf{1}$ to $\mathbf{3}$.
This is the only pair of paths which share both, a common head and a common tail.
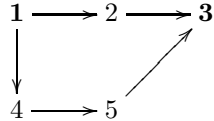
\[\xymatrix{\mathbf{1}\ar[d]\ar[r]&2\ar[r]&\mathbf{3}\\
 4\ar[r]&5\ar[ur]}\]
}
\end{example}
 
The empiric observations made in Examples~\ref{exm3} and \ref{exm4} can be formulated as the
following lemma, whose proof is obvious and thus left to the reader.
 
\begin{lemma}\label{lem5}
Let $Q$ be an admissible orientation of $\tilde{A}_{n}$.
\begin{enumerate}[$($i$)$]
\item\label{lem5.1} If $k>1$, then $\pP^*=\A$.
\item\label{lem5.2} If $k=1$, then $\omega\in\pP^*$ belong to $\A$ if, and only if, it is not maximal. 
Moreover, there are exactly two maximal paths, and these two paths are both, $\mathcal{T}$- and 
$\mathcal{H}$-equivalent.
\end{enumerate}
\end{lemma}

\subsection{Ideals of $\Bbbk Q$}\label{s3.2}

After our observations in Lemma~\ref{lem5}, we can describe all algebra ideals in $\Bbbk Q$.

\begin{proposition}\label{prop6}
Let $Q$ be an admissible orientation of $\tilde{A}_{n}$ and $I$ an algebra ideal in $\Bbbk Q$. 
Then exactly one of the following statements holds:
\begin{enumerate}[$($a$)$]
\item\label{prop6.1} The ideal $I$ is a linearized semigroup ideal.
\item\label{prop6.2} We have $k=1$, and $I=\Bbbk(\omega+a\upsilon)$, where 
$\omega$ and $\upsilon$ are the two different maximal paths and $a\in\Bbbk\setminus\{0\}$.
\end{enumerate}
\end{proposition}

\begin{proof}
If $k>1$, then, by Lemma~\ref{lem5}\eqref{lem5.1}, we have $\pP^*=\A$.
Therefore, in the case $k>1$, from  Proposition~\ref{prop2} it follows that 
$I$ is a linearized semigroup ideal, implying that we are in the situation as described in \eqref{prop6.1}.

It remains to consider the case $k=1$. Clearly, the algebra ideals of the form
$\Bbbk(\omega+a\upsilon)$, for $a\in\Bbbk\setminus\{0\}$, are not linearized semigroup ideals.
Therefore, we may assume that $I$ is not a linearized semigroup ideal and need to show that 
$I$ is of the form as in \eqref{prop6.2}.

If $I$ contains some element that, when written as a linear combination of paths, contains a
non-zero coefficient at some non-maximal path $\upsilon$, then $\upsilon\in I$ by Lemma~\ref{lem1}
as all non-maximal paths belong to $\A$. In this case the unique maximal path $\omega$ such that $\upsilon\leqJ\omega$
also belongs to $I$ as the latter is an ideal. Let $\varpi$ be the other maximal path in $\pP^*$, that
is the one different from $\omega$. If $\varpi\in I$, then from Lemma~\ref{lem1} we immediately have
that $I$ is a linearized semigroup ideal. If $\varpi\not\in I$, then any element of $I$, when 
written as a linear combination of paths, must have a zero coefficient at $\varpi$, for otherwise,
taking into account that $\omega\in I$ and following the proof of Lemma~\ref{lem1} would give us
$\varpi\in I$, a contradiction. In either case, $I$ is a linearized semigroup ideal.

It remains to consider the case when $I$ only contains linear combinations of $\omega$ and $\varpi$.
The argument in the previous paragraph implies that neither $\omega$ nor $\varpi$ can be in $I$
for then $I$ would be a linearized semigroup ideal. Therefore $I$ has to be one-dimensional
and of the form as  in \eqref{prop6.2}. The claim follows.
\end{proof}

It is worth to remember that non-linearized ideals only exist in the case when $k=1$.
In the latter case, when non-linearized ideals exist, the number of such ideals 
is finite only in the case the base field $\Bbbk$ is finite.

\subsection{Graph of an ideal}\label{s3.3}

Let $\Gamma$ be the graph, dual to the undirected graph underlying of $\mathrm{max}(Q)$. In other words,
$\Gamma$ can be described as follows:
\begin{itemize}
\item vertices of $\Gamma$ are maximal paths in $\pP^*$,
\item edges of $\Gamma$ are sinks and sources in $Q$,
\item a sink or source $x$ is the edge between the two maximal paths which have $x$ as the
unique common endpoint.
\end{itemize}
For a linearized semigroup ideal $I\subset \Bbbk Q$, let $\Gamma_I$ be the subgraph of $G$ defined
as follows:
\begin{itemize}
\item vertices of $\Gamma_I$ are exactly the maximal paths in $\pP^*$ which belong to $I$,
\item edges in $\Gamma_I$ are exactly the sinks and sources $x\in Q$ such that $\varepsilon_x$ is in $I$.
\end{itemize}
Note that well-definedness of $\Gamma_I$ is based on the fact that $I$ is an ideal: if
$x$ is a sink or a source in $Q$, then $\varepsilon_x\in $ implies that the two maximal paths 
$\omega,\upsilon$ which have $x$ as the common endpoint are also in $I$. To simplify notation,
we will denote by $V_I$ the set of vertices of $\Gamma_I$ and by $E_I$ the set of edges of $\Gamma_I$.

\begin{example}\label{ex7}
{\rm 
To illustrate the notions defined above, consider the following example:
Let $Q$ be the quiver 
\[\xymatrix{1\ar[d]_{\alpha_1}\ar[r]_{\alpha_2}&2\ar[r]_{\alpha_3}&3\\
4&5\ar[l]_{\alpha_6}&6\ar[l]_{\alpha_5}\ar[u]_{\alpha_4}},\]
and let $I=( \alpha_2,\varepsilon_6)$ be the ideal in $\Bbbk Q$ 
generated by the arrow  $\alpha_2$ and the trivial path $\varepsilon_6$ at $6$. As $I$ is generated
by elements of $\pP$, it is a linearized semigroup ideal. The ideal $I$ contains the 
maximal paths $\alpha_3\alpha_2,\alpha_4$ and $\alpha_6\alpha_5$. The corresponding
graph $\Gamma_I$ is therefore as follows:
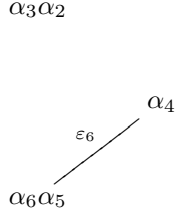
\[\xymatrix{&\alpha_3\alpha_2\\
&&\alpha_4\ar@{-}[dl]_{\varepsilon_6}\\
&\alpha_6\alpha_5}\]
Note that the graph $\Gamma_I$ is disconnected and that $I$ equals the direct sum 
$(\alpha_2)\oplus(\varepsilon_6)$ of two ideals.
}
\end{example}

\begin{remark}\label{rem8}
{\rm
In general, there are many linearized semigroup ideals that share the same graph $\Gamma$. 
In more detail, let $I$ and $J$ be two linearized semigroup ideals in $\Bbbk Q$. 
Then $I\subseteq J$ implies $\Gamma_I\subseteq \Gamma_J$, by definition. However, 
$\Gamma_I=\Gamma_J$ and  $I\subset J$ do not imply $I=J$, in general, as, on the
semigroup level, the ideal $J$ might contain,
compared to $I$, just some extra paths with are neither maximal nor of length $0$. 
}
\end{remark}

The ideals $I$ which are not linearized semigroup ideals are on the form 
$I=\Bbbk(\omega+a\upsilon)$, by Proposition~\ref{prop6}\eqref{prop6.2}. 
In particular, they are one-dimensional and hence indecomposable. 
Our first major observation is the following proposition which says  
that indecomposability of a linearized semigroup ideal $I$ 
is controlled by the connectedness of the corresponding graph $\Gamma_I$.

\begin{theorem}\label{thm9}
Let $I$ be a non-zero linearized semigroup ideal. Then the ideal $I$ is indecomposable 
(as an ideal) if, and only if, the graph $\Gamma_I$ is connected.
\end{theorem}

\begin{proof}
We begin with the ``if'' direction. Assume that $\Gamma_I$ is connected and that
we have a decomposition $I=I_1\oplus I_2$, where $I_1$ and $I_2$ are non-zero ideals. 
Clearly, both $I_1$ and $I_2$ are linearized semigroup ideals. 
Thus $\Gamma_{I_1}$ and $\Gamma_{I_2}$ are well-defined. 
Let $\omega$ be a path in $I$. Then either $\omega\in I_1$ or $\omega\in I_2$. 
In particular, each maximal path in $I$ is either a maximal path in $I_1$
or a maximal path in $I_2$. This gives us a natural decomposition of vertices
in $\Gamma_I$ into a disjoint union of vertices in $\Gamma_{I_1}$ and those in $\Gamma_{I_2}$.
It remains to show that no edge of $\Gamma_I$ can connect a vertex of $\Gamma_{I_1}$
with a vertex of $\Gamma_{I_2}$. Indeed, let $x$ be an edge of $\Gamma_I$ connecting 
a vertex $\omega$ of $\Gamma_{I_1}$ with a vertex $\upsilon$ of $\Gamma_{I_2}$.
Then $\varepsilon_x\in I$ and hence either $\varepsilon_x\in I_1$ or $\varepsilon_x\in I_2$.
In the first case we get $\upsilon\in I_1$, a contradiction. In the second case we
get $\omega\in I_2$, a contradiction. This completes the proof of the ``if'' direction.

For the ``only if'' direction assume that $\Gamma_I$ is disconnected. 
Then $\Gamma_I$ has two or more connected components. Pick one maximal path $\omega_0$ in $I$. 
Let $I_1$ be the linear span of all paths $\upsilon$ in $I$ for which the unique maximal path 
$\omega$  satisfying $\upsilon\leqJ\omega$ is in the same connected component of
$\Gamma_i$ as $\omega_0$. Define $I_2$ to be the span of all paths in $I$ which are not in $I_1$. 
Then $I_1$ and $I_2$ are, by construction, non-zero algebra ideals such that $I=I_1\oplus I_2$.
Therefore $I$ is decomposable. This completes the proof.
\end{proof}  

\subsection{Tensor product of ideals}\label{s3.4}

For two algebra ideals $I$ and $J$ in $\Bbbk Q$, consider the multiplication map 
\begin{displaymath}
\varphi_{I,J}:I\otimes_{\Bbbk Q} J\to IJ,\quad a\otimes b\mapsto ab,\quad
\text{ for all } a\in I_1\text { and }b\in I_2.
\end{displaymath}

\begin{lemma}\label{lem10}
The map $\varphi_{I,J}$ is an isomorphism of $\Bbbk Q$-$\Bbbk Q$-bimodules.
\end{lemma}

\begin{proof}
The map $\varphi_{I,J}$ is a homomorphism of $\Bbbk Q$-$\Bbbk Q$-bimodules by construction. 
The surjectivity of $\varphi_{I,J}$ is clear. Indeed, for a general element $x=\sum_s a_sb_s\in IJ$,
where all $a_s\in I$ and $b_s\in J$, the map $\varphi_{I,J}$ maps  the element 
$\sum_s a_s\otimes b_s\in I\otimes_{\Bbbk Q} J$ to $x$. 
It remains to prove the injectivity of $\varphi_{I,J}$.

Consider the short exact sequence
\begin{equation}\label{eq1}
0\to J\to \Bbbk Q\to \mathrm{Coker}\to 0 
\end{equation}
of $\Bbbk Q$-$\Bbbk Q$-bimodules, where $J\to \Bbbk Q$ is the natural inclusion and
$\Bbbk Q\to \mathrm{Coker}$ is the natural projection. As the algebra $\Bbbk Q$ is hereditary, 
the ideal $I$, being a sub-bimodule of the regular $\Bbbk Q$-$\Bbbk Q$-bimodule $\Bbbk Q$, is
projective as a right $\Bbbk Q$-module. Therefore the functor $I\otimes_{\Bbbk Q}{}_-$ is exact. 
Applying this exact functor to the short exact sequence in Equation~\eqref{eq1} results in the exact sequence
\begin{equation}\label{eq2}
0\to I\otimes_{\Bbbk Q}J\to I\otimes_{\Bbbk Q}\Bbbk Q\to I\otimes_{\Bbbk Q}\mathrm{Coker}\to 0.
\end{equation}
For the middle term in the sequence in Equation~\eqref{eq2}, the corresponding multiplication map 
$\varphi_{I,\Bbbk Q}:I\otimes_{\Bbbk Q}\Bbbk Q\to I$
is, clearly, an isomorphism. Composing this isomorphism with the inclusion 
$I\otimes_{\Bbbk Q}J\hookrightarrow I\otimes_{\Bbbk Q}\Bbbk Q$ from the sequence in Equation~\eqref{eq2}, results in the
map $\varphi_{I,J}$. This shows that $\varphi_{I,J}$ is, indeed, injective. The proof is complete.
\end{proof}

Because of Lemma~\ref{lem10}, we may identify $I\otimes_A J$ with the ideal $IJ\subset A$ via $\varphi_{I,J}$.

\subsection{Graph of ideals vs tensor products of ideals}\label{s3.5}

The following proposition describes the behavior of the graph of an ideal under tensor product.

\begin{proposition}\label{prop11}
For linearized semigroup ideals $I$ and $J$ in $\Bbbk Q$, we have:
\begin{enumerate}[$($i$)$]
\item\label{prop11.1} $\Gamma_{I\otimes_A J}\subset \Gamma_I\cap \Gamma_J$.
\item\label{prop11.2} $E_{I\otimes_A J}=E_I\cap E_J$.
\item\label{prop11.3} If the graph $\Gamma_I\cap \Gamma_J$ has no isolated vertices, then $\Gamma_{I\otimes_A J}= \Gamma_I\cap \Gamma_J$.
\item\label{prop11.4} There is an ideal $J'\subset J$ such that $I\otimes_A J=I\otimes_A J'$, and 
\[\Gamma_{I\otimes_A J}=\Gamma_{I\otimes_A J'}=\Gamma_I\cap \Gamma_{J'}. \] 
If, moreover, $J$ is indecomposable, then so is $J'$.
\end{enumerate}
\end{proposition}

\begin{proof}
We begin by proving claim~\eqref{prop11.1}. Let $\omega$ be a maximal path in $IJ$. 
Then $\omega=\omega_l\omega_r$, for some $\omega_l$ in $I$ and $\omega_r$ in $J$. 
But, as both, $\omega_l\leqJ\omega$ and $\omega_r\leqJ\omega$, it follows that $\omega$ 
belongs to  both, $I$ and $J$. This proves that $V_{I\otimes_A J}\subset V_I\cap V_J$. 
To prove $E_{I\otimes_A J}\subset E_I\cap E_J$, note that the equation 
$\omega\cdot\upsilon=\varepsilon_x$, where $\omega\in\pP^*$ and $\upsilon\in\pP^*$, has the unique 
solution $\omega=\upsilon=\varepsilon_x$. Thus $\varepsilon_x\in IJ$ implies 
that $\varepsilon_x$ is also in both, $I$ and $J$. Claim~\eqref{prop11.1} is proved.

To prove claim~\eqref{prop11.2}, it remains to show that $E_{I\otimes_A J}\supset E_I\cap I_J$. 
Let $\varepsilon_x$ in $E_I\cap E_J$. Then $\varepsilon_x=\varepsilon_x\cdot \varepsilon_x$ is also in 
$E_{I\otimes_A J}$ and claim~\eqref{prop11.2} follows.

To prove  claim~\eqref{prop11.3}, assume that $\Gamma_I\cap \Gamma_J$ has no isolated vertices. 
From claim~\eqref{prop11.2} we have that $E_{I\otimes_A J}=E_I\cap E_J$, and
from claim~\eqref{prop11.1} we have that $\Gamma_{I\otimes_A J}\subset \Gamma_I\cap \Gamma_J$.
These imply that all vertices adjacent to some edge 
in $E_I\cap E_J$ are in $\Gamma_{I\otimes_A J}$. By assumption, all vertices in $V_I\cap V_J$ are
adjacent to some edge in $E_I\cap E_J$. Hence $V_{I\otimes_A J}\supset V_I\cap V_J$. Put together, this 
gives $\Gamma_{I\otimes_A J}= \Gamma_I\cap \Gamma_J$, completing the proof of claim~\eqref{prop11.3}.

It remains to prove claim~\eqref{prop11.4}. Because of claim~\eqref{prop11.3},
we only need to prove claim~\eqref{prop11.4} in the situation when $\Gamma_I\cap G_J$ has isolated vertices. 
Denote by  $\Delta$ the difference graph
$(\Gamma_I\cap \Gamma_J)\setminus \Gamma_{I\otimes_A J}$. Because of the arguments in the previous
paragraph, $\Delta$ consists only of isolated vertices. We denote the set of vertices of
$\Delta$ by $\Delta_0$. If $\Delta$ is empty, claim~\eqref{prop11.4} is 
obvious. Therefore we may assume that $\Delta$ is not empty. 

Let $J'$ be the span of all paths in $J$ 
that can not be extended to any $\omega\in \Delta_0$:
\[J':=\Bbbk\set{\upsilon\in J\cap\pP^*\mid \upsilon\not\leqJ \omega,\text{ for any }\omega\in\Delta_0}.\]
Clearly $J'\subset J$, which implies that $IJ\supset IJ'$. 
Moreover, $J'$ is an ideal, as it is the intersection of the ideals $J$ with the ideal
\begin{displaymath}
\Bbbk\set{\upsilon\in\pP^*\mid \upsilon\not\leqJ \omega, \text{ for any }\omega\in\Delta_0}. 
\end{displaymath}
It remains to show that $IJ\subset IJ'$. 
Let $\omega\in\Delta_0$ and $\upsilon\in J\setminus J'$ be a path such that $\upsilon\leqJ \omega$.
Assume that $\varpi\in I$ is  a path such that $\varpi\upsilon$ is non-zero. Then 
$IJ\ni \varpi\upsilon\leqJ \omega$ (as $\upsilon\leqJ \omega$ and $\omega$ is maximal), 
implying $\omega\in \Gamma_{I\otimes_A J}$. This contradicts our definition of $\Delta$.
Consequently, $I\upsilon=0$, and hence $IJ\subset IJ'$. Put together, we get $IJ=IJ'$.

Assume that $J$ is indecomposable. Then $\Gamma_J$ is connected by Theorem~\ref{thm9}.
Directly from the definitions we have that $\Gamma_J$ is either an unoriented cycle or
a chain. In the first case, $J$ contains $\varepsilon_x$, for any sink or source of $Q$,
together will all maximal paths in $\pP$. Therefore, for any maximal path $\omega$ in $I$, 
we have $\omega=\omega\varepsilon_{\mathbf{t}(\omega)}\in IJ$, as $\mathbf{t}(\omega)$ is a source. 
Hence in this case $\Delta$ is empty. 

It remains to consider the case when $\Gamma_J$ is a chain, say
\begin{displaymath}
\xymatrix{ 
\omega_1\ar@{-}[r]^{x_1}&\omega_2\ar@{-}[r]^{x_2}&\omega_3\ar@{-}[r]^{x_3}&\dots
\ar@{-}[r]^{x_{k-2}}&\omega_{k-1}\ar@{-}[r]^{x_{k-1}}&\omega_k.
} 
\end{displaymath}
We claim that $\Delta_0\subset\{\omega_1,\omega_k\}$, which certainly implies indecomposability of $J'$,
because of the definitions and Theorem~\ref{thm9}. To prove that
$\Delta_0\subset\{\omega_1,\omega_k\}$, assume that $\Delta_0$ contains
$\omega_s$, for some $s\neq 1,k$. Then both, $\varepsilon_{\mathbf{h}(\omega_s)}$
and $\varepsilon_{\mathbf{t}(\omega_s)}$ are in $J$, because of the form of $\Gamma_J$. 
Hence $\omega_s$ belongs both to $I$ and $IJ$,
which contradicts $\omega_s\in \Delta_0$. This completes the proof of claim~\eqref{prop11.4}
and of the whole proposition.
\end{proof}

\subsection{Decomposition of tensor product}\label{s3.6}

\begin{corollary}\label{cor12}
Let $I_1,\cdots, I_m$ be indecomposable linearized semigroup ideals in $\Bbbk Q$.
Then, for ${}_{\Bbbk Q}I_{\Bbbk Q}:={}_{\Bbbk Q}(I_1I_2\cdots I_m)_{\Bbbk Q}\cong 
{}_{\Bbbk Q}(I_1\otimes_{\Bbbk Q} I_2\otimes_{\Bbbk Q}\cdots\otimes_{\Bbbk Q} I_m)_{\Bbbk Q}$, 
the following statements hold.
\begin{enumerate}[$($i$)$]
\item\label{cor12.1} The ideal $I$ is a direct sum of at most $\min(m,2k)$ indecomposable ideals.
\item\label{cor12.2} Each indecomposable summand in $I$ has multiplicity one.
\item\label{cor12.3} Each indecomposable summand in $I$ is a linearized semigroup ideal.
\end{enumerate}
\end{corollary}

\begin{proof}
We note that the ${\Bbbk Q}$-${\Bbbk Q}$-bimodule isomorphism 
$${}_{\Bbbk Q}(I_1I_2\cdots I_m)_{\Bbbk Q}\cong 
{}_{\Bbbk Q}(I_1\otimes_{\Bbbk Q} I_2\otimes_{\Bbbk Q}\cdots\otimes_{\Bbbk Q} I_m)_{\Bbbk Q}$$
follows from Lemma~\ref{lem10} and is given by the multiplication map.
The ideal $I$ is the span of 
all elements on the form $\omega_1\omega_2\cdots\omega_m$, where all  $\omega_i$ are paths. 
Therefore $I$ is a linearized semigroup ideal  and its graph $\Gamma_I$ is well-defined. 

Each indecomposable summand of $I$ corresponds to a non-empty connected component in $\Gamma_I$,
by Theorem~\ref{thm9}. In particular, it is a linearized semigroup ideal and appears in 
$\Gamma_I$ with multiplicity one. The number of such components does not exceed the number of vertices
of $\Gamma_I$ and the latter is equal to $2k$. Let us now use induction on $m$ to prove that the number 
of connected components of  $\Gamma_I$ does not exceed $m$. The case $m=1$ is obvious.

To make the induction step, assume that $J:=I_1I_2\cdots I_{m-1}$ is a sum of at most $m-1$ 
indecomposable summands. Then $\Gamma_{J}$ has at most $m-1$ connected components and they all are,
obviously, disjoint. 

Let $K$ and $K'$ be two indecomposable linearized semigroup ideals such that $KK'\neq 0$.
If either $\Gamma_K$ or $\Gamma_{K'}$ is a cycle, then, from Proposition~\ref{prop11}, we
get the equality $\Gamma_{KK'}=\Gamma_K\cap \Gamma_{K'}$ and hence $KK'$ is indecomposable. 
If both, $\Gamma_K$ and $\Gamma_{K'}$, are chains, there are two possibilities:
\begin{itemize}
\item The intersection $\Gamma_K\cap \Gamma_{K'}$ is again a chain. In this case 
from Proposition~\ref{prop11} it follows that $KK'$ is indecomposable.
\item The intersection $\Gamma_K\cap \Gamma_{K'}$ is a union of two chains. This is
only possible in case $\Gamma_K\cup \Gamma_{K'}$ is the full cycle
(i.e. coincides with  $\Gamma_{\Bbbk Q}$). In this case 
$KK'$ might decompose into a direct sum of two indecomposable summands (in which case each 
summand corresponds to a connected component of $\Gamma_K\cup \Gamma_{K'}$).
\end{itemize}

Now, let us return to $\Gamma_{J}$, which is a disjoint union of connected components. 
From the previous paragraph it follows that, for the connected graph $\Gamma_{I_m}$,
there could exist at most one connected component of $\Gamma_{J}$ such that the
union of this component with $\Gamma_{I_m}$ is a cycle. Only this component can split into
two different components when going from $J$ to $JI_m$. Therefore the number of 
indecomposable direct summands of $JI_m$ is at most $m$. This completes the proof.
\end{proof}

\begin{example}\label{exm13}
{\em
Here is an example which illustrates how the product of two ideals splits into a direct sum.
Let $Q$ be given by:
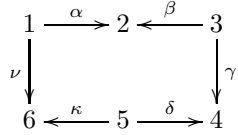
\begin{displaymath}
\xymatrix{ 
1\ar[r]^{\alpha}\ar[d]_{\nu}&2&3\ar[l]_{\beta}\ar[d]^{\gamma}\\
6&5\ar[l]_{\kappa}\ar[r]^{\delta}&4
}
\end{displaymath}
Let $I$ be the ideal of $\Bbbk Q$ generated by $\varepsilon_1,\varepsilon_2,\varepsilon_5,\varepsilon_6$.
Then $\Gamma_I$ is as follows:
\begin{displaymath}
\xymatrix{ 
\beta\ar@{-}[r]^{2}&\alpha\ar@{-}[r]^{1}&\nu\ar@{-}[r]^{6}&\kappa\ar@{-}[r]^{5}&\delta
}
\end{displaymath}
Let $J$ be the ideal of $\Bbbk Q$ generated by $\varepsilon_2,\varepsilon_3,\varepsilon_4,\varepsilon_5$.
Then $\Gamma_J$ is as follows:
\begin{displaymath}
\xymatrix{ 
\alpha\ar@{-}[r]^{2}&\beta\ar@{-}[r]^{3}&\gamma\ar@{-}[r]^{4}&\delta\ar@{-}[r]^{5}&\kappa
}
\end{displaymath}
Note that $\Gamma_I\cup \Gamma_J$ is the full cycle:
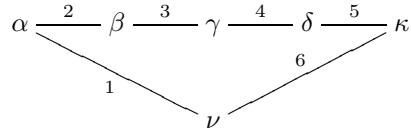
\begin{displaymath}
\xymatrix{ 
\alpha\ar@{-}[r]^{2}&\beta\ar@{-}[r]^{3}&\gamma\ar@{-}[r]^{4}&\delta\ar@{-}[r]^{5}&\kappa\\
&&\nu\ar@{-}[rru]^{6}\ar@{-}[llu]^{1}&&
}
\end{displaymath}
The ideal $IJ$ is generated by $\varepsilon_2$ and $\varepsilon_5$ and $\Gamma_{IJ}$ is as follows:
\begin{displaymath}
\xymatrix{ 
\alpha\ar@{-}[r]^{2}&\beta&&\delta\ar@{-}[r]^{5}&\kappa
}
\end{displaymath}
The ideal $IJ$ is thus a direct sum of two indecomposable summands, the first one is the
ideal generated by $\varepsilon_2$ and the second one is the ideal generated by $\varepsilon_5$.
} 
\end{example}

Denote by $\mathcal{S}$ the set of (isomorphism classes of) indecomposable ideals in $\Bbbk Q$.
If $k>1$, then all projective $\Bbbk Q$-modules are multiplicity free. In this case 
the regular $\Bbbk Q$-$\Bbbk Q$-bimodule ${}_{\Bbbk Q}\Bbbk Q_{\Bbbk Q}$ is multiplicity free
and therefore all subbimodules of this bimodule are uniquely determined by their composition multiplicities. 
This implies that $\mathcal{S}$ is a finite set, in the case $k>1$.
The set $\mathcal{S}$ has the natural structure of a {\em multisemigroup} given, for $I,J\in \mathcal{S}$, by:
\begin{displaymath}
I\star J:=\{K\,:\, K\text{ is an indecomposable summand of }IJ\}. 
\end{displaymath}
We refer to \cite{KM} for more details on multisemigroups and to \cite[Section~5]{GM2} and \cite[Section~3]{Zh2} 
for  more details on multisemigroups of ideals in path algebras.

Let $I$ be an indecomposable  linearized semigroup ideal. We define its {\em width} $\mathfrak{w}(I)$ as 
$\mathfrak{w}(I)=|V_I|$. For $0\leq m\leq 2k$, we define $\mathcal{S}_{\leq m}$ to be the subset of $\mathcal{S}$ 
consisting of  all linearized ideals $I$ such that $\mathfrak{w}(I)\leq m$ or $I=\Bbbk Q$.

\begin{corollary}\label{cor14}
Let $0\leq m\leq 2k$. Then we have the following:
\begin{enumerate}[$($i$)$]
\item\label{cor14.1} The set $\mathcal{S}_{\leq m}$ inherits from $\mathcal{S}$ the structure of 
a multisemigroup with identity. 
\item\label{cor14.2} The multisemigroup $\mathcal{S}_{\leq m}$ is a semigroup if, and only if, $m\leq k$. In this case $\mathcal{S}_{\leq m}$ is a monoid with a zero element.
\item\label{cor14.3} If $k\geq 2$, then $\mathcal{S}_{\leq 2k}=\mathcal{S}$.
\end{enumerate}
\end{corollary}

\begin{proof}
To prove that $\mathcal{S}_{\leq m}$ is a semigroup, we need to show that $\mathcal{S}_{\leq m}$ 
is closed under tensor products inside $\mathcal{S}$. But this follows immediately from the facts
that the tensor product of two linearized semigroup ideals is a linearized semigroup ideal, 
and that $\Gamma_{I\otimes_A J}\subset \Gamma_I\cap \Gamma_J$, see Proposition~\ref{prop11}\eqref{prop11.1}. 
This implies claim~\eqref{cor14.1}.

Assume that $m\leq k$. Without loss of generality, we may assume that $I\neq \Bbbk Q\neq J$. 
By Proposition~\ref{prop11}\eqref{prop11.4}, we have
$I\otimes_A J=I\otimes_A J'$, for some indecomposable linearized semigroup ideal $J'\subset J$ 
such that  $\Gamma_{I\otimes_A J'}=\Gamma_I\cap \Gamma_{J'}$. Note that $J'\subset J$ implies that 
$\mathfrak{w}(J')\leq \mathfrak{w}(J)\leq m$. Therefore both $\Gamma_I$ and $\Gamma_{J'}$ are chains
and the intersection $\Gamma_I\cap \Gamma_{J'}$ contains at most one connected component, as $m\leq k$. 
Consequently, $IJ=IJ'$ is indecomposable.

If $m>k$, then the obvious generalization of Example~\ref{exm13} proves existence of 
two indecomposable $I$ and $J$ such that $IJ$ decomposes. This establishes claim~\eqref{cor14.2}.

Claim~\eqref{cor14.3} follows directly from the definitions, completing the proof. 
\end{proof}


\subsection{$\Bbbk Q$ as stair-shaped algebras}\label{s3.7}

We say that a matrix algebra $B\subset\mathrm{Mat}_{r\times r}(\Bbbk)$ is \emph{stair-shaped} if 
it consists of all matrices of the following fixed form
\[\left(\begin{array}{ccccccccccc}
\boldsymbol{*} &\cdots & * &&&&&&&&\\
&\ddots&\vdots&&&&&&&&\\
&&*&&&&&&&&\\
&&\vdots&\ddots&&&&&&&\\
&&*&\cdots &*&\cdots&*&&&&\\
&&&&&\ddots&\vdots&&\\
&&&&&&*&&\\
&&&&&&&\ddots&&&\\
&&&&&&&&*&&\\
&&&&&&&&\vdots&\ddots&\\
&&&&&&&&*&\cdots&\boldsymbol{*}
\end{array}\right).\]
Here the sizes of triangles are arbitrary, all entries different from 
$*$ are supposed to be zero and the value of the two bold entries
(the north-west and the south-east corners) are the same.
Also, here the last triangle may also point north-east instead of south-west. Note that we require 
the two ends of the main diagonal to have the same value. Let $k'$ denote the number of triangles
on this picture which point south-west. 

\begin{example}\label{exm15}
{\rm  Here are two examples of stair-shaped algebras, both with $k'=1$:
\begin{displaymath}
\left\{ 
\left(\begin{array}{ccc}a&b&0\\0&c&0\\0&d&a\end{array}\right)\,:\,a,b,c,d\in\Bbbk
\right\}
\text{ and }
\left\{ 
\left(\begin{array}{cccc}a&b&0&0\\0&c&0&0\\0&d&e&f\\0&0&0&a\end{array}\right)\,:\,a,b,c,d,e,f\in\Bbbk
\right\}.
\end{displaymath}
For the first algebra, the last triangle points south-west.
For the second algebra, the last triangle points north-east.
}
\end{example}

\begin{proposition}\label{prop16}
{\hspace{2mm}}

\begin{enumerate}[$($i$)$]
\item\label{prop16.1}  Let $Q$ be as in Subsection~\ref{s3.1}.
Then $\Bbbk Q$ is isomorphic to a stair-shaped algebra  
with $k'=k$ triangles pointing north-east (and hence the last triangle pointing south-west).
\item\label{prop16.2} Every stair-shaped algebra with the last triangle pointing south-west
is isomorphic to $\Bbbk Q$, for some $Q$ as in Subsection~\ref{s3.1}.
\item\label{prop16.3} Let $Q$ be as in Subsection~\ref{s3.1} and $\alpha$ and $\beta$ be
two composable arrows of $Q$ (i.e. $\beta\alpha\in\pP^*$). Then $\Bbbk Q/(\beta\alpha)$ 
is isomorphic to a stair-shaped algebra with  $k'=k$ and the last triangle north-east.
\item\label{prop16.4} Every stair-shaped algebra with the last triangle pointing north-east
is isomorphic to $\Bbbk Q/(\beta\alpha)$, for some $Q$ be as in Subsection~\ref{s3.1} and 
for some composable arrows $\alpha$ and $\beta$ in $Q$.
\end{enumerate}
\end{proposition}

\begin{proof}
We start by proving claim~\eqref{prop16.1}.
Start from a fixed source $x$ in $Q$, and walk around the underlying circle of $Q$
(in some fixed direction). The zero 
length paths $\varepsilon_y$, for $y$ different from $x$, then correspond to diagonal 
matrix units, in the circular order. The element $\varepsilon_x$ corresponds to the sum of
the south-east and the north-west diagonal matrix units. In this way, each path from $y$ 
to $z$, where $y,z\neq x$, corresponds to the matrix unit in the intersection of the row indexed by $z$
with the column indexed by $y$. If $z=x$ (in particular, this is the case if $y=x$,
as $y$ is a source), then a path from $y$ to $z$ corresponds to a matrix unit
in the first row and the column indexed by $x$ if there are no sinks between $x$ and $y$ 
in the direction of our walk. Otherwise, our path corresponds to a matrix unit
in the last column. This defines the isomorphism which proves claim~\eqref{prop16.1}.
Note that $r=n+1$. For example, the first algebra in Example~\ref{exm15} is obtained in this
way from the Kronecker algebra, that is an admissible orientation of the $\tilde{A}_{2}$ diagram:
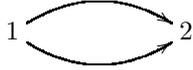
\begin{displaymath}
\xymatrix{1\ar@/^1pc/[rr]\ar@/_1pc/[rr]&&2} 
\end{displaymath}

Given a stair-shaped algebra $B$ with the last triangle pointing south-west, we can reverse the
correspondence described in the previous paragraph and thus read off an orientation
$Q$ of the affine Dynkin diagram of type $\tilde{A}_{r-1}$ such that $B$ is isomorphic to $\Bbbk Q$.
This proves claim~\eqref{prop16.2}.

Claims~\eqref{prop16.3} and \eqref{prop16.4} are proved similarly. The only difference is that,
instead of staring at a fixed source, one starts at the vertex $x=\mathbf{h}(\alpha)=\mathbf{t}(\beta)$.
We leave the details to the reader.
\end{proof}

\section{Combinatorics}\label{s4}

\subsection{Ideals and Dyck paths}\label{s4.1}

Recall that a {\em Dyck path} of semilength $n$ is a path on an $(n+1)\times (n+1)$-grid from the 
north-west to the south-east corner which never goes below, but is allowed to touch, the diagonal 
that connects these two corners. The number of Dyck paths of semilength $n$ is given by the $n$-th 
{\em Catalan number} $C_{n}:=\frac{1}{n+1}\binom{2n}{n}$, see \cite[Chapter~6]{St}.
Dyck paths of semilength $n$ correspond bijectively to nilpotent ideals in the associative algebra 
$T_n$ of upper triangular $n\times n$-matrices with coefficients in $\Bbbk$. To state the correspondence, 
we can view an $n\times n$-matrix as a ``filling'' of an $(n+1)\times (n+1)$-grid. Then, the correspondence 
mentioned above associates to a Dyck path the ideal consisting of all matrices which have zero entries 
below the path. 

Similarly one shows that all ideals in $T_n$ are in 
bijection with Dyck path of semilength $n+1$ and hence are enumerated by $C_{n+1}$. 
All such ideals are indecomposable and form a semigroup with respect to
multiplication (alternatively, with respect to tensor product). This 
gives a categorification of the Catalan monoid of all order decreasing and oder preserving
transformations of a finite chain and was studied in detail in \cite{GM1}. Note that
the algebra of all upper triangular $n\times n$-matrices is isomorphic to the path algebra
of a uniform orientation of the (finite type) Dynkin diagram of type $A_n$.

\begin{example}\label{exm17}
{\rm 
The following picture shows the $10$-dimensional ideal $I$ generated by the matrix units $e_{3,3}$ and $e_{4,5}$  
in the algebra of upper triangular  $5\times 5$-matrices:
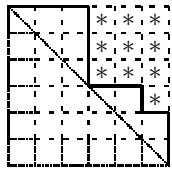
\begin{center}
\begin{picture}(80,80)
\dashline{2}(20,20)(20,70)
\dashline{2}(30,20)(30,70)
\dashline{2}(40,20)(40,70)
\dashline{2}(50,20)(50,70)
\dashline{2}(60,20)(60,70)
\dashline{2}(70,20)(70,70)
\dashline{2}(20,20)(70,20)
\dashline{2}(20,30)(70,30)
\dashline{2}(20,40)(70,40)
\dashline{2}(20,50)(70,50)
\dashline{2}(20,60)(70,60)
\dashline{2}(20,70)(70,70)
\dottedline(10,10)(10,70)
\dottedline(10,10)(70,10)
\dottedline(10,70)(70,10)
\dottedline(20,10)(20,20)
\dottedline(30,10)(30,20)
\dottedline(40,10)(40,20)
\dottedline(50,10)(50,20)
\dottedline(60,10)(60,20)
\dottedline(70,10)(70,20)
\dottedline(10,20)(20,20)
\dottedline(10,30)(20,30)
\dottedline(10,40)(20,40)
\dottedline(10,50)(20,50)
\dottedline(10,60)(20,60)
\dottedline(10,70)(20,70)
\thicklines
\drawline(10,70)(40,70)
\drawline(40,40)(40,70)
\drawline(40,40)(60,40)
\drawline(60,30)(60,40)
\drawline(60,30)(70,30)
\drawline(70,10)(70,30)
\put(42.50,42.50){$*$}
\put(52.50,42.50){$*$}
\put(62.50,42.50){$*$}
\put(42.50,52.50){$*$}
\put(52.50,52.50){$*$}
\put(62.50,52.50){$*$}
\put(42.50,62.50){$*$}
\put(52.50,62.50){$*$}
\put(62.50,62.50){$*$}
\put(62.50,32.50){$*$}
\end{picture}
\end{center}
Potentially non-zero entires for elements in $I$ are indicated by $*$.
The boundaries of the matrix are indicated by dashed lines. Dotted lines
indicate an ``extension'' of the matrix necessary to include non-nilpotent
ideals. The dotted diagonal is the boundary which is not allowed to 
be crossed by Dyck paths. The Dyck path which corresponds to $I$ is the thick 
solid path on the picture.
}
\end{example}

An alternative way to describe all ideals in $T_n$ is to use the terminology of
{\em generalized Dyck paths}. A generalized Dyck describes the south-west boundary of an ideal.
For the ideal $I$ in Example~\ref{exm17}, the corresponding generalized Dyck paths
is given by the $\bullet$-entires in Figure~\ref{fig0}.
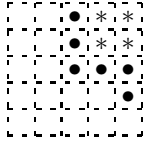
\begin{figure}
\begin{center}
\begin{picture}(80,80)
\dashline{2}(20,20)(20,70)
\dashline{2}(30,20)(30,70)
\dashline{2}(40,20)(40,70)
\dashline{2}(50,20)(50,70)
\dashline{2}(60,20)(60,70)
\dashline{2}(70,20)(70,70)
\dashline{2}(20,20)(70,20)
\dashline{2}(20,30)(70,30)
\dashline{2}(20,40)(70,40)
\dashline{2}(20,50)(70,50)
\dashline{2}(20,60)(70,60)
\dashline{2}(20,70)(70,70)
\put(42.50,42.50){$\bullet$}
\put(52.50,42.50){$\bullet$}
\put(62.50,42.50){$\bullet$}
\put(42.50,52.50){$\bullet$}
\put(52.50,52.50){$*$}
\put(62.50,52.50){$*$}
\put(42.50,62.50){$\bullet$}
\put(52.50,62.50){$*$}
\put(62.50,62.50){$*$}
\put(62.50,32.50){$\bullet$}
\end{picture}
\end{center}
\caption{An example of a generalized Dyck path}\label{fig0}
\end{figure}
As we see from this example, a generalized Dyck path is described by the following rules:
\begin{itemize}
\item It is  a path on an $n\times n$-grid (instead of $(n+1)\times (n+1)$-grid).
\item It starts at some point on the upper row of the grid.
\item Each step is either to the right or down.
\item It ends at some point on the rightmost column of the grid.
\item It never goes under the main diagonal.
\end{itemize}
Note that a generalized Dyck path may be empty (this corresponds to the zero ideal of $T_n$).
Clearly, the number of generalized Dyck paths on a $n\times n$-grid is $C_{n+1}$, the
same as the number of ideals in $T_n$.

In \cite{GM2},  Catalan combinatorics plays important role in enumeration of indecomposable ideals
in path algebras associated to arbitrary orientations of finite type $A$ Dynkin diagrams.
Stair-shaped algebras correspond to orientations of affine type $A$ Dynkin diagrams. Therefore
it is natural to expect that Catalan numbers and some generalizations of Dyck paths should
also appear in the study of ideals for stair-shaped algebras. We describe the corresponding
combinatorics of certain cylindrical generalizations of Dyck paths in this section and
the enumeration of ideal in the next section. The combinatorics we describe generalizes the one described
in \cite{GM2}. We also note that a version of cylindrical Dyck paths appeared in \cite{HR}.

For a quiver $Q$ as in Subsection~\ref{s3.1}, we fix a realization of the algebra
$\Bbbk Q$ as a stair-shaped algebra of the form described in Subsection~\ref{s3.7},
see Proposition~\ref{prop16}. In this realization we have $k$ triangles pointing
north-east and $k$ triangles pointing south-west. Going from top left to bottom right,
let the lengths of the diagonal parts of these triangles be $i_1,i_2,\dots,i_{2k}$.
Note that the triangles corresponding to $i_1,i_3,\dots$ point  north-east while
the triangles corresponding to $i_2,i_4,\dots$ point south-west.
The vector $\mathbf{i}=(i_1,i_2,\dots,i_{2k})$ will be called the {\em signature}
of $Q$. The corresponding stair-shaped algebra will be denoted by $B_{\mathbf{i}}$.
Note that, subtracting vector $(1,1,\dots,1)$ from the signature gives
a vector which records lengths of all maximal paths in $Q$. Note also that $n=i_1+i_2+\dots+i_{2k}-2k$.

Let $I$ be a linearized semigroup ideal in $\Bbbk Q$. Then $I$ has a basis consisting of the
matrix units in the stair-shaped realization of $\Bbbk Q$. Define the {\em boundary} of
$I$ as the set of all matrix units which 
\begin{itemize}
\item belong to the south-west boundary of the ideal $I$ in all triangles pointing north-east;
\item belong to the north-east boundary of the ideal $I$ in all triangles pointing south-west.
\end{itemize}
The boundary of the zero ideal is empty. This definition obvious agrees with the above
notion of the boundary for an ideal in $T_n$. 

\begin{example}\label{exm18}
{\rm
On Figure~\ref{fig1} one finds an example
of an ideal (whose potentially non-zero components are indicated by $*$) and its
boundary (whose elements are indicated by $\bullet$) for the stair-shaped algebra with 
signature $(2,4,4,3)$. The remaining potentially non-zero places of the algebra are indicated by $\circ$.
Note that, in the example given by Figure~\ref{fig1}, the boundary is connected.
However, in general, it can also be disconnected.
} 
\end{example}

\begin{figure}
\begin{center}
\begin{picture}(120,120)
\dashline{2}(10,10)(10,110)
\dashline{2}(20,10)(20,110)
\dashline{2}(30,10)(30,110)
\dashline{2}(40,10)(40,110)
\dashline{2}(50,10)(50,110)
\dashline{2}(60,10)(60,110)
\dashline{2}(70,10)(70,110)
\dashline{2}(80,10)(80,110)
\dashline{2}(90,10)(90,110)
\dashline{2}(100,10)(100,110)
\dashline{2}(110,10)(110,110)
\dashline{2}(10,10)(110,10)
\dashline{2}(10,20)(110,20)
\dashline{2}(10,30)(110,30)
\dashline{2}(10,40)(110,40)
\dashline{2}(10,50)(110,50)
\dashline{2}(10,60)(110,60)
\dashline{2}(10,70)(110,70)
\dashline{2}(10,80)(110,80)
\dashline{2}(10,90)(110,90)
\dashline{2}(10,100)(110,100)
\dashline{2}(10,110)(110,110)
\put(12.50,102.50){$\circ$}
\put(22.50,102.50){$\circ$}
\put(22.50,92.50){$\circ$}
\put(22.50,82.50){$\circ$}
\put(22.50,72.50){$\bullet$}
\put(22.50,62.50){$*$}
\put(32.50,82.50){$\circ$}
\put(32.50,72.50){$\bullet$}
\put(32.50,62.50){$\bullet$}
\put(42.50,72.50){$\circ$}
\put(42.50,62.50){$\bullet$}
\put(52.50,62.50){$\bullet$}
\put(62.50,62.50){$\bullet$}
\put(62.50,52.50){$\circ$}
\put(72.50,62.50){$\bullet$}
\put(72.50,52.50){$\bullet$}
\put(72.50,42.50){$\bullet$}
\put(82.50,62.50){$*$}
\put(82.50,52.50){$*$}
\put(82.50,42.50){$\bullet$}
\put(82.50,32.50){$\bullet$}
\put(82.50,22.50){$\bullet$}
\put(82.50,12.50){$*$}
\put(92.50,22.50){$\bullet$}
\put(92.50,12.50){$\bullet$}
\put(102.50,12.50){$\circ$}
\end{picture}
\end{center} 
\caption{An ideal and its boundary}\label{fig1} 
\end{figure}
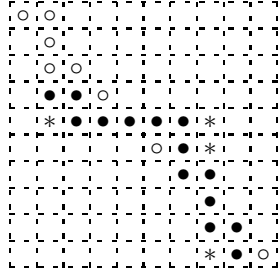

An axiomatic description of a boundary of an ideal is collected in the following definition.

\begin{definition}\label{def19}
{\rm 
Given an $n\times n$ grid and a signature $\mathbf{i}=(i_1,\dots,i_{2k})$
such that $i_1+i_2+\dots+i_{2k}-2k=n$, a {\em cylindrical generalized Dyck path $\pi$
of signature $\mathbf{i}$} is a possibly empty collection of paths on this grid 
satisfying the following conditions:
\begin{enumerate}[$($a$)$]
\item\label{def19.1} Each step of $\pi$ belongs to one of the triangles of $B_{\mathbf{i}}$.
\item\label{def19.2} The intersection of $\pi$ with each of the triangles of
$B_{\mathbf{i}}$ is a generalized Dyck path for that triangle.
\item\label{def19.3} The north-west corner belongs to $\pi$ if and only if 
the south-east corner does.
\end{enumerate} 
}
\end{definition}

Note that condition of Definition~\ref{def19}\eqref{def19.3} corresponds to the cylindrical nature of our 
generalized Dyck paths. The connection between ideals and cylindrical generalized Dyck path is given 
by the following statement. A cylindrical generalized Dyck path is called {\em connected} provided that
it is a connected path, after identification of the north-west and the south-east corners.

\begin{proposition}\label{prop20}
In the situation as above, we have:
\begin{enumerate}[$($i$)$] 
\item\label{prop20.1} The boundary of any ideal $I$ in $\Bbbk Q$ is a cylindrical generalized Dyck path,
which we will denote by $\pi(I)$. This correspondence between ideals and 
cylindrical generalized Dyck paths is bijective. 
\item\label{prop20.2} Indecomposable ideals correspond to connected cylindrical generalized Dyck paths.
\item\label{prop20.3} Nilpotent ideals correspond to cylindrical generalized Dyck paths that 
have empty intersection with the main diagonal.
\end{enumerate}
\end{proposition}

\begin{proof}
Claim~\eqref{prop20.1} follows directly from definitions. Claim~\eqref{prop20.3} follows from the
fact that an ideal is nilpotent if and only if it does not contain any of the primitive
idempotents of $\Bbbk Q$ (these idempotents are given by the diagonal elements in $B_{\mathbf{i}}$).

Claim~\eqref{prop20.2} follows from Theorem~\ref{thm9} and the observation that triangles intersecting
with $\pi(I)$ are exactly the vertices of $\Gamma_I$ and the diagonal elements of $\pi(I)$ are exactly
the edges of $\Gamma_I$.
\end{proof}

\subsection{Enumeration of ideals}\label{s4.2}

Assume that the quiver  $Q$ is as in Subsection~\ref{s3.1}. Let $\mathbf{i}$ be the signature of $Q$.
For simplicity, we define $i_{2k+1}:=i_1$, $i_{2k+2}:=i_2$ and so on. Similarly, we define
$i_0:=i_{2k}$, $i_{-1}:=i_{2k-1}$ and so on. For each $s\in\mathbb{Z}$, we also define
$l_s:=i_s-1$. Then $l_s$ is the length of a maximal path in $Q$ and the vector 
$(l_1,l_2,\dots,l_{2k})$ provides a complete list of lengths of maximal paths in $Q$. 
Note that $i_s\geq 2$, for all $s$.

Recall also that the number of ideals in the algebra $T_n$ of upper triangular $n\times n$ matrices
is exactly $C_{n+1}$. All ideals in $T_n$ are indecomposable (as $T_n$-$T_n$-bimodules). 
Further, the number of nilpotent ideals in $T_n$ equals $C_{n}$.

\begin{theorem}\label{thm21}
Let $Q$ be as in Subsection~\ref{s3.1} and of signature $\mathbf{i}$. 
\begin{enumerate}[$($i$)$]
\item\label{thm21.1} The number of indecomposable ideals in $\Bbbk Q$ is 
\begin{multline*}
\prod_{s=1}^{2k}C_{i_s-1}+\\+
\sum_{t=1}^{2k}\sum_{s=1}^{2k-1}(C_{i_t}-C_{i_t-1})
(C_{i_{t+s-1}}-C_{i_{t+s-1}-1})\prod_{r=2}^{s-1}C_{i_{t+r-1}-1}+\\
+\sum_{t=1}^{2k}(C_{i_t+1}-2C_{i_t}+C_{i_t-1}-1).
\end{multline*}
\item\label{thm21.2} The number of nilpotent ideals in $\Bbbk Q$ is $\displaystyle\prod_{s=1}^{2k}C_{i_s}$. 
\item\label{thm21.3} The number of nilpotent indecomposable ideals in $\Bbbk Q$ is 
$\displaystyle\left(\sum_{s=1}^{2k}C_{i_s}\right)-2k$.
\end{enumerate}
\end{theorem}

\begin{proof}
An ideal $I$ of $\Bbbk Q$ is nilpotent and indecomposable if and only if $\pi(I)$ is connected
and contains no diagonal elements, see Proposition~\ref{prop20}. Therefore $I$ is a nilpotent ideal
inside one of the triangular parts of $B_{\mathbf{i}}$. This implies claim~\eqref{thm21.3} using 
the rule of sum and taking into account that zero ideals are not indecomposable. 
This also implies claim~\eqref{thm21.2} using the rule of product. So, it remains to
prove claim~\eqref{thm21.1}.

The graph $\Gamma_I$ of an indecomposable ideal $I$ is connected by Theorem~\ref{thm9}.
Let $s$ be the number of vertices in $\Gamma_I$. If $s=0$, then $I$ is the zero ideal
and hence is not indecomposable. If $s=1$, then $\pi(I)$ is a generalized Dyck path 
of the corresponding $T_m$ which does not contain any endpoints of the triangle.
The number of such Dyck paths is counted in \cite[Proposition~27(ii)]{GM2} and equals 
\begin{displaymath}
C_{m+1}-2C_{m}+C_{m-1}-1. 
\end{displaymath}

If $s=2k$, then $\Gamma_I$ is a full circle. Then $I$ is given by choosing, in each triangle
of the algebra $B_{\mathbf{i}}$, an ideal of the corresponding $T_m$ which contains the
diagonal endpoints of this triangle. This correspond to a Dyck path of semilength $m-2$.
Therefore we get the total number 
\begin{displaymath}
\prod_{s=1}^{2k}C_{i_s-1} 
\end{displaymath}
of ideals $I$ for which $\Gamma_I$ is a full circle.

If $1<s<2k$, then $\Gamma_I$ is a chain with two endpoints. 
Let $t\in\{1,2,\dots,2k\}$ be one of the endpoints of this
chain. We assume that the chain goes from $t$ in the circular order of our walk (which was used
to identify $\Bbbk Q$ with $B_{\mathbf{i}}$). Then the second endpoint of the chain
will be $t+s-1$, modulo $2k$. The cylindrical generalized Dyck paths $\pi(I)$
intersects only the triangles corresponding to vertices of $\Gamma_I$. If this vertex is not
one of the endpoints, then the corresponding component of $\pi_I$ is counted 
in the previous paragraph.
For the endpoints, the corresponding component of $\pi(I)$ should contain one end of the
corresponding triangle for $T_r$, but not the other. The number of such components equals
$C_r-C_{r-1}$, see \cite[Lemma~29]{GM2}.

Therefore the number of $I$ such that $\Gamma_I$ is a chain with two endpoints is given by
\begin{displaymath}
\sum_{t=1}^{2k}\sum_{s=1}^{2k-1}(C_{i_t}-C_{i_t-1})
(C_{i_{t+s-1}}-C_{i_{t+s-1}-1})\prod_{r=2}^{s-1}C_{i_{t+r-1}-1} 
\end{displaymath}
Claim~\eqref{thm21.1} now follows by putting all these formulae together.
\end{proof}

The formulae in Theorem~\ref{thm21} become much nicer if one assumes that all maximal
paths in $Q$ have the same length.

\begin{corollary}\label{cor22}
Let $Q$ be as in Subsection~\ref{s3.1} and of signature $\mathbf{i}$.
Assume that this signature is constant in the sense that $i_1=i_2=\dots =i_{2k}=:i$.
Then we have the following:
\begin{enumerate}[$($i$)$]
\item\label{cor22.1} If $i>2$, then the number of indecomposable ideals in $\Bbbk Q$ is
\[2k\left(C_{i+1}-2C_{i}+C_{i-1}-1+(C_{i}-C_{i-1})^2\frac{C_{i-1}^{2k-1}-1}{C_{i-1}-1}\right)+C_{i-1}^{2k}.\]
\item\label{cor22.2} If $i=2$, then the number of indecomposable ideals is $4k^2+1$.
\item\label{cor22.3} The number of nilpotent ideals in $\Bbbk Q$ is $C_{i}^{2k}$.
\item\label{cor22.4} The number of nilpotent indecomposable ideals in $\Bbbk Q$ is $2k(C_{i}-1)$.
\end{enumerate}
\end{corollary}


\subsection{Multiplication of indecomposable ideals and a combinatorial description for decomposability of the product}\label{s4.3}

In this subsection we study multiplication in the multisemigroup $\mathcal{S}$. To start with,
we give explicit formulae for multiplication with those algebra ideals in $\Bbbk Q$ 
which are not linearized semigroup ideals.

\begin{lemma}\label{lem23}
Assume that $k=1$ and $\omega$ and $\upsilon$ are the two different longest paths in $Q$. Let $x$ be the source
in $Q$ and $y$ be the sink in $Q$. Let $a\in\Bbbk\setminus\{0\}$. Then, for any indecomposable algebra ideal 
$I$ in $\Bbbk Q$, we have
\begin{displaymath}
(\omega+a\upsilon)\cdot I\cong
\begin{cases}
(\omega+a\upsilon), &\varepsilon_x\in I;\\
0, & \text{ otherwise;} 
\end{cases}
\end{displaymath}
and 
\begin{displaymath}
I\cdot (\omega+a\upsilon)\cong
\begin{cases}
(\omega+a\upsilon), &\varepsilon_y\in I;\\
0, & \text{ otherwise.} 
\end{cases}
\end{displaymath}
In particular, $(\omega+a\upsilon)\cdot(\omega+b\upsilon)=0$, for any $b\in\Bbbk\setminus\{0\}$.
\end{lemma}

\begin{proof}
The formula $(\omega+a\upsilon)\cdot(\omega+b\upsilon)=0$, for any $b\in\Bbbk\setminus\{0\}$,
follows directly from the observation that $x\neq y$. It remains to consider the case when $I$
is a linearized semigroup ideal. If $\varepsilon_x\in I$, then 
$(\omega+a\upsilon)\cdot I=(\omega+a\upsilon)$ follows directly from the fact that 
$(\omega+a\upsilon)$ is $1$-dimensional. If $\varepsilon_x\not\in I$, then $I$ contains no
paths which terminate at $x$, as $x$ is a source of $Q$. Therefore 
$(\omega+a\upsilon)\cdot I=0$ in this case. The second formula is proved similarly.
\end{proof}

From Lemma~\ref{lem23} it follows that, for $k=1$, 
algebra ideals which are not linearized semigroup ideals form a nilpotent ideal in
$\mathcal{S}$. Now let us consider multiplication of linearized semigroup ideals.
Let us split all indecomposable linearized semigroup ideals into three types.
\begin{enumerate}[$($I$)$]
\item Ideals of type I are those ideals whose graph is the full circle.
\item Ideals of type II are those ideals whose graph is a chain with at least two vertices.
\item Ideals of type III are those ideals whose graph consists of a single vertex.
\end{enumerate}
The arguments in the proof of Corollary~\ref{cor12} imply the following:

\begin{lemma}\label{lem24}
Let $I$ and $J$ be two indecomposable linearized semigroup ideals in $\Bbbk Q$. If
$I\cdot J$ is not indecomposable, then both $I$ and $J$ are of type {\rm II} and 
$\Gamma_I\cup \Gamma_J$ is the full circle.
\end{lemma}

\begin{proof}
This is proved during the proof of  Corollary~\ref{cor12}.
\end{proof}

Let us now give more explicit description of  when, for two  indecomposable linearized semigroup ideals
$I$ and $J$ of type II, their product $IJ$ decomposes into a direct sum of two ideals.

\begin{lemma}\label{lem25}
Let $I$ and $J$ be two indecomposable linearized semigroup ideals in $\Bbbk Q$ of type {\rm II} such that $\Gamma_I\cup \Gamma_J$ is the full circle. Then
$I\cdot J$ decomposes into a direct sum of two non-zero components if and only if the following condition
is satisfied: Each isolated vertex $\omega$ in $\Gamma_I\cap \Gamma_J$ can be written as
$\omega=\alpha\beta$, for some $\alpha\in I$ and $\beta\in J$.
\end{lemma}

\begin{proof}
The condition of the lemma is, by definition, equivalent to the fact that 
$\Gamma_{IJ}=\Gamma_I\cap \Gamma_J$, which implies the statement.
\end{proof}

One nice sufficient condition for decomposition is the following:

\begin{corollary}\label{cor26}
Let $I$ and $J$ be two indecomposable linearized semigroup ideals in $\Bbbk Q$ of type {\rm II} such that $\Gamma_I\cup \Gamma_J$ is the full circle. 
Assume that, for each isolated vertex $\omega$ in $\Gamma_I\cap \Gamma_J$, the edge $x$ in 
$\Gamma_J$ adjacent to $\omega$ is a source.  Then
$I\cdot J$ decomposes into a direct sum of two non-zero components.
\end{corollary}

\begin{proof}
If the condition of the lemma is satisfied, then $\varepsilon_x\in J$ and $\omega\in I$,
moreover, $\omega\varepsilon_x\neq 0$. Therefore the necessary claim follows from
Lemma~\ref{lem25}.
\end{proof}

Compared to the situation studied in \cite{GM2}, we see that, in our case, the multisemigroup 
$\mathcal{S}$ is no longer a semigroup. In some cases the multiplication is indeed multivalued.
We would like to finish with an explicit combinatorial criterion, in terms of Dyck path
combinatorics, for when the product of two indecomposable linearized semigroup ideals decomposes
into a sum of two non-zero components.

Let $\pi$ be a generalized Dyck path as in Subsection~\ref{s4.1}. We read $\pi$ starting from
the top row and going to the right column. Define the {\em initial slope} $\mathbf{is}(\pi)$
of $\pi$ as the number of the column from which $\pi$ starts. Define the {\em terminal slope}
$\mathbf{ts}(\pi)$ of $\pi$ as the number of the row in which $\pi$ terminates. For example,
for the generalized Dyck path $\pi$ in Figure~\ref{fig0}, we have $\mathbf{is}(\pi)=3$
and $\mathbf{ts}(\pi)=4$. Now we are ready to formulate our combinatorial characterization
of decomposability.

\begin{theorem}\label{thm27}
Let $I$ and $J$ be two indecomposable linearized semigroup ideals in $\Bbbk Q$ of type {\rm II} such that $\Gamma_I\cup \Gamma_J$ is the full circle. 
Then $I\cdot J$ decomposes into a direct sum of two non-zero components if and only if the following 
condition is satisfied: For each isolated vertex $\omega$ in $\Gamma_I\cap \Gamma_J$
such that the edge $x$ in $\Gamma_J$ adjacent to $\omega$ is a sink,
we have 
\begin{equation}\label{eq5}
\mathbf{ts}(\pi(J)_{\omega})\geq\mathbf{is}(\pi(I)_{\omega}), 
\end{equation}
where $\pi(I)_{\omega}$ denotes the intersection of $\pi(I)$ with the triangle corresponding to $\omega$
and similarly for $\pi(J)_{\omega}$ (our normalization is such that $\mathbf{is}(\pi(J)_{\omega})=1$
and $\mathbf{ts}(\pi(I)_{\omega})=\mathfrak{l}(\omega)+1$).
\end{theorem}

\begin{proof}
We have to take a look at the algebra $T_m$, where $m=\mathfrak{l}(\omega)+1$, which 
corresponds to $\omega$ in out stair-shaped realization of $\Bbbk Q$, namely, we have to
investigate when $(T_m\cap I)(T_m\cap J)$ is non-zero. Our normalization
is chosen such that $\mathbf{ts}(\pi(J)_{\omega})< \mathfrak{l}(\omega)+1$ and
$\mathbf{is}(\pi(I)_{\omega})>1$. The ideal $I$ contains the matrix unit $e_{1,\mathbf{is}(\pi(I)_{\omega})}\in T_m$.
The ideal $J$ contains the matrix unit $e_{\mathbf{ts}(\pi(J)_{\omega}),m}\in T_m$.

If the inequality \eqref{eq5} is satisfied, then the ideal $J$ also contains 
$e_{\mathbf{is}(\pi(I)_{\omega}),m}\in T_m$ because $J$ is an ideal and because of \eqref{eq5}.
Then we have 
$e_{1,\mathbf{is}(\pi(I)_{\omega})}e_{\mathbf{is}(\pi(I)_{\omega}),m}\neq 0$ which implies that the
$\omega$-component of $IJ$ is non-zero.

If the inequality \eqref{eq5} is not satisfied, then, for any matrix unit $e_{a,b}\in I\cap T_m$ and 
any matrix unit $e_{c,d}\in J\cap T_m$, we have $b>c$ and hence $e_{a,b}e_{c,d}=0$. This implies that 
the $\omega$-component of $IJ$ is zero. The claim of the theorem follows.
\end{proof}

Department of Mathematics, Uppsala University, Box. 480,
SE-75106, Uppsala, SWEDEN, email: {\tt love.forsberg\symbol{64}math.uu.se}


\begin{thebibliography}{99999999}
\bibitem[BFK]{BFK} J.~Bernstein, I.~Frenkel, M.~Khovanov. A categorification of the Temperley-Lieb 
algebra and Schur quotients of $U(\mathfrak{sl}_2)$ via projective and Zuckerman functors. Selecta Math. 
(N.S.) {\bf 5} (1999), no. 2, 199--241. 
\bibitem[CM]{CM} A.~Chan, V.~Mazorchuk. Diagrams and discrete extensions for finitary $2$-representations.
Preprint arXiv:1601.00080.
\bibitem[CR]{CR} J.~Chuang, R.~Rouquier. Derived equivalences for symmetric groups and 
$\mathfrak{sl}_2$-categorification. Ann. of Math. (2) {\bf 167} (2008), no. 1, 245--298. 
\bibitem[Fo]{Fo} L. Forsberg. Multisemigroups with multiplicities and complete ordered semi-rings. Beitr Algebra Geom (2016), 1--22.
\bibitem[Gr]{Gr} J.~A.~Green. On the structure of semigroups. Ann. of Math. (2) {\bf 54}, (1951). 163--172. 
\bibitem[GM1]{GM1} A.-L. Grensing, V. Mazorchuk. Categorification of the Catalan monoid.
Semigroup Forum \textbf{89} (2014), no. 1, 155--168.
\bibitem[GM2]{GM2} A.-L. Grensing, V. Mazorchuk. Categorification using dual projection functors.
Preprint arXiv:1501.00095. To appear in Commun. Contemp. Math.
\bibitem[HR]{HR} J.~Hartwig, D.~Rosso. Cylindrical Dyck paths and the Mazorchuk-Turowska equation. 
J. Algebraic Combin. {\bf 44} (2016), no. 1, 223--247. 
\bibitem[Kh]{Kh} M. Khovanov. A categorification of the Jones polynomial. Duke Math. J. {\bf 101} (2000), no. 3, 359--426.
\bibitem[KiM1]{KM1} T. Kildetoft, V. Mazorchuk. Parabolic projective functors in type $A$. Adv. Math. {\bf 301} (2016), 785--803.
\bibitem[KiM2]{KM2} T. Kildetoft, V. Mazorchuk. Special modules over positively based
algebras. Documenta Math. {\bf 21} (2016) 1171--1192.
\bibitem[KMMZ]{KMMZ} T.~Kildetoft, M.~Mackaay, V.~Mazorchuk, J.~Zimmermann. Simple transitive
$2$-representations of small quotients of Soergel bimodules. Preprint arXiv:1605.01373.
\bibitem[KM]{KM} G. Kudryavsteva, V. Mazorchuk. On multisemigroups. Port. Math. {\bf 72} (2012), no 1, 47--80.
\bibitem[MaMa]{MaMa} M.~Mackaay, V.~Mazorchuk. Simple transitive $2$-representations for some
$2$-subcategories of Soergel bimodules. J. Pure Appl. Algebra  {\bf 221} (2017), no. 3, 565--587.
\bibitem[MMMT]{MMMT} M.~Mackaay, V.~Mazorchuk, V.~Miemietz and D.~Tubbenhauer.
Simple transitive $2$-representations via (co)algebra $1$-morphism. Preprint arXiv:1612.06325.
\bibitem[MT]{MT} M.~Mackaay, D.~Tubbenhauer. Two-color Soergel calculus and simple transitive
$2$-rep\-re\-sen\-tations. Preprint arXiv:1609.00962.
\bibitem[Ma]{Ma} Lectures on Algebraic Categorification. QGM Master Class Series. European Mathematical Socielty (EMS), Z{\"u}rich, 2012.
\bibitem[MM1]{MM1} V.~Mazorchuk, V.~Miemietz. Cell $2$-representations of finitary
$2$-categories. Compositio Math. {\bf 147} (2011), 1519--1545.
\bibitem[MM2]{MM2} V.~Mazorchuk, V.~Miemietz. Additive versus abelian $2$-representations of 
fiat $2$-ca\-te\-go\-ri\-es. Moscow Math. J. {\bf 14} (2014), no. 3, 595--615.
\bibitem[MM3]{MM3} V.~Mazorchuk, V.~Miemietz. Endomorphisms of cell $2$-representations. 
IMRN. {\bf 2016} (2016), no 24, 7471--7498.
\bibitem[MM4]{MM4} V.~Mazorchuk, V.~Miemietz. Morita theory for finitary $2$-categories. 
Quantum Topol. {\bf 7} (2016), no 1, 1--28.
\bibitem[MM5]{MM5} V.~Mazorchuk, V.~Miemietz. Transitive $2$-representations of finitary $2$-categories. 
Trans. Amer. Math. Soc. {\bf 368} (2016), no 11, 7623--7644.
\bibitem[MM6]{MM6} V.~Mazorchuk, V.~Miemietz. Isotypic faithful $2$-representations of 
$\mathcal{J}$-simple fiat $2$-categories. Preprint arXiv:1408.6102. 
\bibitem[MMZ]{MMZ} V. Mazorchuk, V. Miemietz, X.~Zhang. Characterisation and applications of
$\Bbbk$-split bimodules. Preprint  arXiv:1701.03025.
\bibitem[MZ]{MZ} V.~Mazorchuk, X.~Zhang. Simple transitive $2$-representations for two non-fiat
$2$-categories of projective functors. Preprint arXiv:1601.00097.
\bibitem[Ro]{Ro} R.~Rouquier. 2-Kac-Moody algebras.
Preprint arXiv:0812.5023.
\bibitem[St]{St} R.~Stanley. Enumerative combinatorics. Vol. 2. Cambridge Studies in Advanced 
Mathematics, {\bf 62}. Cambridge University Press, Cambridge, 1999. 
\bibitem[Xa]{Xa} Q.~Xantcha. Gabriel $2$-quivers for finitary $2$-categories. J. Lond. Math. Soc.
(2) {\bf 92} (2015), no. 3, 615--632.
\bibitem[Zh1]{Zh1} X.~Zhang. Duflo involutions for $2$-categories associated to tree quivers.
J. Algebra Appl. {\bf 15} (2016), no. 3, 1650041, 25 pp.
\bibitem[Zh2]{Zh2} X.~Zhang. Simple transitive $2$-representations and Drinfeld center
for some finitary $2$-categories. Preprint arXiv:1506.02402.
\bibitem[Zi]{Zi} J.~Zimmermann. Simple transitive $2$-representations of Soergel bimodules
in type $B_2$. J. Pure Appl. Algebra {\bf 221} (2017), no. 3, 666--690.
\end{thebibliography}
\end{document}